\documentclass[a4,11pt]{amsart}
\usepackage{amsfonts}
\usepackage{mathrsfs}
\usepackage{amsthm,amsxtra}
\usepackage{amssymb}
\usepackage{amsmath}
\usepackage{amscd}
\usepackage{enumerate}
\usepackage{hyperref}
\usepackage{comment}
\usepackage{t1enc}
\usepackage[mathscr]{eucal}
\usepackage{indentfirst}
\usepackage{graphicx}
\graphicspath{ {images/} }
\usepackage[margin=2.9cm]{geometry}
\usepackage{tikz-cd}
\usepackage{tikz}

\theoremstyle{plain}

\newtheorem{theorem}{Theorem}[section]
\newtheorem{proposition}[theorem]{Proposition}
\newtheorem{lemma}[theorem]{Lemma}
\newtheorem{corollary}[theorem]{Corollary}

\theoremstyle{definition}
\newtheorem{definition}[theorem]{Definition}

\theoremstyle{remark}

\DeclareMathOperator{\Cor}{Cor}

\DeclareMathOperator{\Aut}{Aut}

\DeclareMathOperator{\Out}{Out}

\DeclareMathOperator{\C}{C}

\DeclareMathOperator{\PSL}{PSL}

\DeclareMathOperator{\Id}{Id}

\DeclareMathOperator{\Sz}{Sz}
\DeclareMathOperator{\OO}{\mathcal{O}}
\DeclareMathOperator{\CC}{\mathcal{C}}
\DeclareMathOperator{\PG}{\mathbf{P}\mathbf{G}}
\DeclareMathOperator{\FF}{\mathbb{F}}

\makeatletter
\DeclareFontFamily{OMX}{MnSymbolE}{}
\DeclareSymbolFont{MnLargeSymbols}{OMX}{MnSymbolE}{m}{n}
\SetSymbolFont{MnLargeSymbols}{bold}{OMX}{MnSymbolE}{b}{n}
\DeclareFontShape{OMX}{MnSymbolE}{m}{n}{
    <-6>  MnSymbolE5
   <6-7>  MnSymbolE6
   <7-8>  MnSymbolE7
   <8-9>  MnSymbolE8
   <9-10> MnSymbolE9
  <10-12> MnSymbolE10
  <12->   MnSymbolE12
}{}
\DeclareFontShape{OMX}{MnSymbolE}{b}{n}{
    <-6>  MnSymbolE-Bold5
   <6-7>  MnSymbolE-Bold6
   <7-8>  MnSymbolE-Bold7
   <8-9>  MnSymbolE-Bold8
   <9-10> MnSymbolE-Bold9
  <10-12> MnSymbolE-Bold10
  <12->   MnSymbolE-Bold12
}{}

\let\llangle\@undefined
\let\rrangle\@undefined
\DeclareMathDelimiter{\llangle}{\mathopen}%
                     {MnLargeSymbols}{'164}{MnLargeSymbols}{'164}
\DeclareMathDelimiter{\rrangle}{\mathclose}%
                     {MnLargeSymbols}{'171}{MnLargeSymbols}{'171}
\makeatother

\title[flag transitive geometries with trialities and no dualities for $\Sz(q)$]{flag transitive geometries with trialities and no dualities coming from Suzuki Groups}

\author{Dimitri Leemans}\thanks{This research was made possible thanks to an Action de Recherche Concert\'ee grant from the Communaut\'e Fran\c caise Wallonie-Bruxelles.}
\address{Dimitri Leemans, Universit\'e Libre de Bruxelles, D\'epartement de Math\'ematique, C.P.216 - Alg\`ebre et Combinatoire, Boulevard du Triomphe, 1050 Brussels, Belgium, Orcid number 0000-0002-4439-502X.}
\curraddr{}
\email{leemans.dimitri@ulb.be}
\urladdr{}

\author{Klara Stokes}
\address{Klara Stokes, Department of Mathematics and Mathematical Statistics, Ume\aa\; University,
901 87 Ume\aa, Sweden, Orcid number 0000-0002-5040-2089.}
\email{klara.stokes@umu.se}

\author{Philippe Tranchida}
\address{Philippe Tranchida, Universit\'e Libre de Bruxelles, D\'epartement de Math\'ematique, C.P.216 - Alg\`ebre et Combinatoire, Boulevard du Triomphe, 1050 Brussels, Belgium, Orcid number 0000-0003-0744-4934.}
\curraddr{}
\email{tranchida.philippe@gmail.com}
\urladdr{}

\date{\today}
{}
\begin{document}

\maketitle
\begin{abstract}
Recently, Leemans and Stokes constructed an infinite family of incidence geometries admitting trialities but no dualities from the groups $\PSL(2,q)$ (where $q=p^{3n}$ with $p$ a prime and $n>0$ a positive integer). Unfortunately these geometries are not flag transitive.
In this paper, we construct the first infinite family of incidence geometries of rank three that are flag transitive and have trialities but no dualities.
These geometries are constructed using chamber systems of Suzuki groups $\Sz(q)$ (where $q=2^{2e+1}$ with $e$ a positive integer and $2e+1$ is divisible by 3)  
and the trialities come from field automorphisms. We also construct an infinite family of regular hypermaps with automorphism group $\Sz(q)$ that admit trialities but no dualities.
\end{abstract}

\section{Introduction}

In~\cite{LeemansStokes2019}, Leemans and Stokes showed how to construct coset geometries having trialities by using a group $G$ that has outer automorphisms of order three. They gave examples for the smallest simple Suzuki group $\Sz(8)$, some small $\PSL(2,q)$ groups (with $q\in \{8, 27, 64, 125, 512\}$), and the Mathieu group $M_{11}$.
They used this technique in~\cite{leemans2022incidence} to construct the first infinite family of coset geometries admitting trialities but no dualities. These geometries are of rank four, and constructed using regular maps of Wilson class III for the groups $\PSL(2,q)$, where $q = p^{3n}$ for a prime $p$ and a positive integer $n > 0$.
These geometries are not flag transitive. The group $\PSL(2,q)$ has two orbits on their chambers.

The geometries for $\Sz(8)$ mentioned in~\cite{LeemansStokes2019} are known to be flag transitive thanks to~\cite{Leemans1998} where Leemans classified all thin flag transitive geometries having a group $\Sz(8)$ as type-preserving automorphism group.
Among the 183 thin geometries found for $\Sz(8)$, four of them admit trialities but no dualities.
Nevertheless, in~\cite{LeemansStokes2019}, Leemans and Stokes point out that they do not know if, in the general case, the regular hypermaps they construct are flag transitive coset geometries.
The goal of this article is to prove that at least some of these hypermaps are indeed flag transitive coset geometries, therefore providing the first infinite family of flag transitive geometries having trialities but no dualities.

Using chamber systems, we give a geometric description of rank three flag transitive coset geometries $\Gamma(G,(G_i)_{i \in \{0,1,2\}})$ admitting trialities but no dualities. Here $G = \Sz(q)$ is the Suzuki group over the finite field of order $q$ (where $q = 2^{2e+1}$ for some positive integer $e$ such that $2e+1$ is divisible by 3) and the $G_i$'s are suitably chosen dihedral subgroups of $G$. These geometries all have the following diagram.

\begin{center}

    \begin{tikzpicture}
    
   \filldraw[black] (-2,0) circle (2pt)  node[anchor=north]{};
   \filldraw[black] (2,0) circle (2pt)  node[anchor=north]{};
    \draw (-2,0) -- (2,0) node [midway, below = 3pt, fill=white]{$ \; \; \; \; \; 5 \; \; \; \; \; $};
       \filldraw[black] (-2,0) circle (2pt)  node[anchor=north]{};
   \filldraw[black] (0,3.46410161514) circle (2pt)  node[anchor=north]{};
    \draw (-2,0) -- (0,3.46410161514) node [midway, left = 4pt, fill=white]{$  5  $};
    \draw (2,0) -- (0,3.46410161514) node [midway, right = 4pt, fill=white]{$ 5 \ $};
    \end{tikzpicture}
\end{center}

Most of the work in this article generalizes to constructions of similar geometries with diagrams with labels any integer $m$ dividing one of $q-1$, $q+\sqrt{2q}+1$ or $q-\sqrt{2q}+1$ instead of 5. At the moment, we can only prove that these geometries are flag transitive for $m = 5$. 
For any integer $m$ dividing one of $q-1$, $q+\sqrt{2q}+1$ or $q-\sqrt{2q}+1$ and such that $m \equiv 2 (\bmod\; 3)$, we nevertheless show that they are regular hypermaps, see Corollary \ref{cor:hypermaps}. These hypermaps also admit trialities but no dualities, in the sense of operations on hypermaps (see \cite{jones2010hypermap}).

Hypermaps in the context of Suzuki groups were already investigated in \cite{downs2016mobius}, where the authors give a formula to count the number of regular and orientably regular maps and hypermaps with automorphism group $\Sz(q)$.

The paper is organised as follows.
In Section~\ref{sec:prelim}, we recall the basic properties of the Suzuki groups and the notions in incidence geometry needed to understand this paper. We also recall the definition of an hypermap.
In Section~\ref{sec:main}, we construct the geometries mentioned above and show that they are residually connected, flag transitive, and that they admit trialities and no dualities.
We conclude the paper with a series of remarks in Section~\ref{sec:conclusion}.

\section{Preliminaries}\label{sec:prelim}

We start by an introduction to the Suzuki groups $\Sz(q)$ and some of their properties. We then define incidence geometries, chamber systems and hypermaps.

\subsection{Suzuki groups}\label{sec:Suzuki}

Here we introduce the Suzuki groups, one of the 18 infinite families of finite simple groups. We mainly follow the exposition in \cite{luneburg2012translation}, where many more details can be found.

Let $\FF = \FF(q)$ be the field of order $q = 2^{2e+1}$ for some integer $e \geq 1$. Let $\Phi \colon \FF \to \FF$ be the Frobenius automorphism sending $x$ to $x^2$ and let $\theta$ be the automorphism $ x \to x^r$ with $r = 2^{e+1}$, so that $\theta^2 = \Phi$. Let ${\sf P} = \PG(3,\FF)$ be the $3$-dimensional projective space over $\FF$ and let $(X_0:X_1:X_2:X_3)$ be homogeneous coordinates for $\sf P$.

Let $E$ be the plane of $\sf P$ defined by the equation $X_0 = 0$ and set $U = (0:1:0:0)$. Let 
\begin{equation*}
    x = X_2X_0^{-1} \quad\quad\quad\quad y = X_3X_0^{-1} \quad\quad\quad\quad z = X_1X_0^{-1}
\end{equation*}
be affine coordinates on the affine space ${\sf P}_E = \sf P$ \textbackslash$E$.
Finally, define $\OO$ to be the set of points of $\sf P$ consisting of $U$ together with all the points of ${\sf P}_E$ that satisfy the equation
\begin{equation*}
    z = xy + x^{\theta +2}+y^\theta
\end{equation*}
This set $\OO$ is called a Suzuki-Tits ovoid. The Suzuki group $G := \Sz(q)$ is then the group of projectivities of $\sf P$ that preserve $\OO$.  

The definition of ovoid is due to Tits \cite{Tits1962}. Let $\OO$ be a set of points of some projective space $\Sigma$. Then $\OO$ is called an {\em ovoid} if every line of $\Sigma$ intersects $\OO$ in at most two points, and if, for each point $p$, the union of the set of lines of $\Sigma$ intersecting $\OO$ in precisely $p$ is a hyperplane. It was proven by Tits that the set $\OO$ we defined here is indeed an ovoid under this definition.

We will mostly understand the Suzuki group $G$ by its action on $\OO$. We very briefly state some of the properties of this action.
The group $G$ acts doubly transitively on $\OO$ and only the identity element fixes $3$ points of $\OO$. Let $S$ be a Sylow $2$-subgroup of $G$. Then $S$ is contained in its normalizer $N_G(S)$, a Frobenius group which is the stabilizer $G_P$ of some point $P$ of $\OO$ and has cardinality $q^2(q-1)$. The center $Z(S)$ of $S$ is $Z(S) = \{\gamma \in S \mid \gamma^2 = 1 \}$ and is of cardinality $q$. All elements of even order of $G$ are conjugate to some element in $S$, and therefore have a unique fixed point. The stabilizer $G_{P,Q}$ of two points $P$ and $Q$ of $\OO$ is a cyclic group of order $q-1$, which is a Frobenius complement of $G_P$.

The cardinality of $G$ is $(q^2 +1)q^2(q-1)$, which can be deduced from the double transitivity together with $|G_{P,Q}| = q-1$.

The Frobenius automorphism $\Phi$ also acts on $\OO$. Conjugating elements of $G$ by the action of $\Phi$ gives outer automorphisms of $G$. In fact, $\Out(G)$ is cyclic of order $2e+1$ and is generated by the action of $\Phi$. 


The structure of the maximal subgroups of $G$ is very well understood. 
For two groups $H$ and $K$, the notation $H : K$ denotes a semi-direct product of $H$ by $K$. Here are the facts we will need about maximal subgroups, the proofs can be found in \cite{Suzuki1962}.

\begin{theorem}\label{thm:maximalsubgroups}
Let $\alpha_q = 2^{2e+1} + 2^{e+1} +1$ and $\beta_q =  2^{2e+1} - 2^{e+1} +1$.
The maximal subgroups of $\Sz(q)$ are all conjugate to one of the following.
\begin{enumerate}
    \item A Suzuki subgroup $\Sz(q_0)$ where $\FF_{q_0} \subset \FF_q$ is a field extension with $q_0 = 2^{2e_0+1}$ for some positive integer $e_0$ and $q_0$ is maximal for this property,
    \item A dihedral group $D_{2(q-1)}$,
    \item A group isomorphic to $C_{\alpha_q} : C_4$,
    \item A group isomorphic to $C_{\beta_q} : C_4$.
\end{enumerate}
Moreover, the cyclic subgroups $C_{q-1}, C_{\alpha_q}$ and $C_{\beta_q}$ all intersect trivially.
\end{theorem}

We conclude this section by a technical lemma that will be of crucial importance later.

\begin{lemma}\label{lem:conjClasses}\cite[Lemma 24.3]{luneburg2012translation}
    Let $A$ be the set of involutions of $G$ fixing a point $P \in \OO$ and let $B$ be the set of involutions of $G$ fixing $Q \in \OO$. Then, for any $\omega \in A, \zeta,\zeta'\in B$ we have that $\omega\zeta$ is never conjugated to $\omega\zeta'$.
\end{lemma}

\subsection{Incidence and coset geometries}
To their core, most of the geometric objects of interest to mathematicians are composed of elements together with some relation between them. This very general notion is made precise by the notion of an incidence system, or an incidence geometry. For a more detailed introduction to incidence geometry, we refer to~\cite{buekenhout2013diagram}.

    A triple $\Gamma = (X,*,t)$ is called an \textit{incidence system} over $I$ if
    \begin{enumerate}
        \item $X$ is a set whose elements are called the \textit{elements} of $\Gamma$,
        \item $*$ is a symmetric and reflexive relation (called the \textit{incidence relation}) on $X$, and
        \item $t$ is a map from $X$ to $I$, called the \textit{type map} of $\Gamma$, such that distinct elements $x,y \in X$ with $x * y$ satisfy $t(x) \neq t(y)$. 
    \end{enumerate}
Elements of $t^{-1}(i)$ are called the elements of type $i$.
The \textit{rank} of $\Gamma$ is the cardinality of the type set $I$.
A \textit{flag} in an incidence system $\Gamma$ over $I$ is a set of pairwise incident elements. The type of a flag $F$ is $t(F)$, that is the set of types of the elements of $F.$ A \textit{chamber} is a flag of type $I$. An incidence system $\Gamma$ is an \textit{incidence geometry} if all its maximal flags are chambers.

Let $F$ be a flag of $\Gamma$. An element $x\in X$ is {\em incident} to $F$ if $x*y$ for all $y\in F$. The \textit{residue} of $\Gamma$ with respect to $F$, denoted by $\Gamma_F$, is the incidence system formed by all the elements of $\Gamma$ incident to $F$ but not in $F$. The \textit{rank} of the residue $\Gamma_F$ is equal to rank$(\Gamma)$ - $|F|$.

The \textit{incidence graph} of $\Gamma$ is a graph with vertex set $X$ and where two elements $x$ and $y$ are connected by an edge if and only if $x * y$. Whenever we talk about the distance between two elements $x$ and $y$ of a geometry $\Gamma$, we mean the distance in the incidence graph of $\Gamma$ and simply denote it by $d_\Gamma(x,y)$, or even $d(x,y)$ if the context allows.

Let $\Gamma = \Gamma(X,*,t)$ be an incidence geometry over the type set $I$. A {\em correlation} of $\Gamma$ is a bijection $\phi$ of $X$ respecting the incidence relation $*$ and such that, for every $x,y \in X$, if $t(x) = t(y)$ then $t(\phi(x)) = t(\phi(y))$. If, moreover, $\phi$ fixes the types of every element (i.e $t(\phi(x)) = t(x)$ for all $x \in X$), then $\phi$ is said to be an {\em automorphism} of $\Gamma$. The \emph{type} of a correlation $\phi$ is the permutation it induces on the type set $I$. A correlation of type $(i,j)$ is called a duality if it has order $2$ and a correlation of type $(i,j,k)$ is called a triality if it has order $3$. The group of all correlations of $\Gamma$ is denoted by $\Cor(\Gamma)$ and the automorphism group of $\Gamma$ is denoted by $\Aut(\Gamma)$. Remark that $\Aut(\Gamma)$ is a normal subgroup of $\Cor(\Gamma)$ since it is the kernel of the action of $\Cor(\Gamma)$ on $I$.

If $\Aut(\Gamma)$ is transitive on the set of chambers of $\Gamma$ then we say that $\Gamma$ is {\em flag transitive}. If moreover, the stabilizer of a chamber in $\Aut(\Gamma)$ is reduced to the identity, we say that $\Gamma$ is {\em simply transitive} or {\em regular}.

Incidence geometries can be obtained from a group $G$ together with a set $(G_i)_{i \in I}$ of subgroups of $G$ as described in~\cite{Tits1957}. 
    The \emph{coset geometry} $\Gamma(G,(G_i)_{i \in I})$ is the incidence geometry over the type set $I$ where:
    \begin{enumerate}
        \item The elements of type $i \in I$ are right cosets of the form $G_i \cdot g$, $g \in G$.
        \item The incidence relation is given by non empty intersection. More precisely, the element $G_i \cdot g$ is incident to the element $G_j \cdot k$ if and only if $i\neq j$ and $G_i \cdot g \cap G_j \cdot k \neq \emptyset$.
    \end{enumerate}
\subsection{Chamber systems}
The concept of chamber system was invented by Tits in~\cite{tits1982}.
A \textit{chamber system} over $I$ is a pair $\CC=(C,\{\sim_i,i\in I\})$ consisting of a set $C$, whose members are called \textit{chambers}, and a collection of equivalence relations $\sim_i$ on $C$, indexed by $i \in I$.
Two chambers $c$ and $d$ are called $i-$\textit{adjacent} if $c \sim_i d$. For $i\in I$, each $\sim_i$-equivalence class is called an \textit{i-panel}.
The \textit{rank} of $\CC$ is $|I|$. The chamber system $\CC$ is called \textit{thin} if every $i-$panel is of size exactly $2$.

A \textit{weak homomorphism} $\varphi \colon (C,\{\sim_i,i\in I\}) \to (C',\{\sim_{i'},i'\in I\})$ of chamber systems over $I$ is a map $\varphi \colon C \to C'$ for which a permutation $\pi$ of $I$ can be found such that, for all $c,d \in C$, the relation $c \sim_i d$ implies $\varphi(c) \sim_{\pi(i)} \varphi(d)$. If $\pi = \Id$, the weak homomorphism $\varphi$ is said to be a \textit{homomorphism}.
A bijective homomorphism whose inverse is an homomorphism is called an \textit{isomorphism} and an isomorphism from $\CC$ to $\CC$ is called an \textit{automorphism} of $\CC$. We denote by $\Aut(\CC)$ the group of all automorphisms of $\CC$

The \textit{graph} of $\CC$ is the graph whose vertices are the chambers of $\CC$ and where two chambers $c,d$ are connected by an edge if there exists an $i\in I$ such that $c \sim_i d$.
The chamber system $\CC$ is \textit{connected} if the graph of $\CC$ is connected.

For $J \subset I$, we denote by $\sim_J$ the union of all $\sim_j$ with $j\in J$. A connected component of $(\CC,\sim_J)$ is called a \textit{J-cell} of $\CC$. For $i\in I$, the $(I$ \textbackslash $\{i\})$-cells are called $i$-objects of $\CC$.

If $\CC$ is a chamber system, the \textit{incidence system of $\CC$}, denoted by $\Gamma(\CC)$, is the incidence system over $I$ determined as follows. Its $i$-elements, for $i\in I$ are the pairs $(x,i)$ with $x$ an $i$-object of $\CC$; two elements $(x,k)$ and $(y,l)$ of $\Gamma(\CC)$ are incident if and only if $x \cap y \neq \emptyset$ in $\CC$, i.e., $x$ and $y$ have a chamber in common.

A chamber system $\CC$ over $I$ is \textit{residually connected} if, for every subset $J$ of $I$ and every system of $j$-objects $Z_j$, one for each $j \in J$, with the property that any two have a non-empty intersection, it follows that $\bigcap_{j \in J} Z_j$ is an $(I$ \textbackslash$J$)-cell.

Let $\psi \colon G \to \Aut(\CC)$ be a representation of a group $G$ on a chamber system $\CC$. When $G$ is transitive on the chambers of $\CC$, we say that $G$ is \textit{chamber transitive} on $\CC$. We also say that $\CC$ is \textit{chamber transitive} if $\Aut(\CC)$ is chamber transitive.

Let $G$ be a group, $B$ a subgroup, $(G^{(i)})_{i\in I}$ a system of subgroups of $G$ with $B \subset G^{(i)}$. The \textit{coset chamber system} of $G$ on $B$ with respect to $(G^{(i)})_{i\in I}$, denote by $\CC(G,B,(G^{(i)})_{i\in I})$, has the chamber set consisting of all cosets $gB, g\in G$, and $i$-adjacency determined by $gB \sim_i hB$ if and only if $gG^{(i)} = hG^{(i)}$.

For $i\in I$, the group $G^{(i)}$ is called the \textit{standard parabolic subgroup} of type $I \textbackslash \{i\}$. The subgroup $B$ is called the \textit{Borel subgroup} of $G$. In the cases we will explore in this article, $B$ will always be the identity subgroup.

\begin{proposition}\cite[Proposition 3.6.4]{buekenhout2013diagram}\label{prop:CosetChamberSystem}
    If $\psi \colon G \to \Aut(\CC)$ is a chamber transitive representation of $G$ on a chamber system $\CC$ over $I$, then, for every chamber $c$ of $\CC$, the canonical representation of $G$ on $\CC(G,B,(G^{(i)})_{i\in I})$ is equivalent to $\psi$, where $G^{(i)}$ is the stabilizer of the $i$-cell containing $c$.
\end{proposition}

This proposition allows us to work with chamber transitive chamber systems in a group theoretical way. The notion of connectedness and residual connectedness can then also be translated into group theoretical language.

\begin{lemma}\cite[Lemma 3.6.7]{buekenhout2013diagram} \label{lem:chanmberSystemConnectedness}
    Let $G$ be a group, $B$ a subgroup of $G$, $(G^{(i)})_{i\in I}$ a system of subgroups of $G$ with $B \leq G^{(i)}$. The chamber system $\CC(G,B,(G^{(i)})_{i\in I})$ is connected if and only if $G$ is generated by the subgroups $G^{(i)}$, $i\in I$.
\end{lemma}

For $G$ a group with a system of subgroups $(G^{(i)})_{i\in I}$ and for $J \subset I$, we write $G^{(J)} = \langle G^{(j)} \mid j\in J \rangle$.
We use the following theorem. Observe that in~\cite{buekenhout2013diagram}, the hypothesis that the chamber system is connected is missing.
\begin{theorem}\cite[Theorem 3.6.9]{buekenhout2013diagram} \label{thm:chamberRC}
    Let $I$ be a finite index set. Suppose that $G$ is a group with a system of subgroups $(G^{(i)})_{i\in I}$ and that $B$ is a subgroup of $\bigcap_{i \in I} G^{(i)}$. If the chamber system $\CC(G,B,(G^{(i)})_{i\in I})$ is connected, the following statements are equivalent.
    \begin{enumerate}
        \item The chamber system $\CC(G,B,(G^{(i)})_{i\in I})$ is residually connected.
        \item For all $J,K,L \subseteq I$, we have $G^{(L)} \cap G^{(J)}G^{(K)} = G^{(L\cap J)}G^{(L\cap K)}$.
        \item For all $J,K,L \subseteq I$, we have $G^{(J)}G^{(L)} \cap G^{(K)}G^{(L)} = G^{(J\cap K)}G^{(L)}$.
    \end{enumerate}
\end{theorem} 
Finally, if we restrict ourselves to residually connected incidence geometries and residually connected chamber system, we can go back and forth between the two notions freely.

\begin{theorem}\cite[Theorem 3.4.6]{buekenhout2013diagram}\label{thm:Correspondence} 
Let $\mathcal{G}(I)$ be the collection of all residually connected geometries over $I$ and let $\CC(I)$ be the collection of all residually connected chamber systems over $I$.
    There is a bijective homomorphism preserving structures between the collections $\mathcal{G}(I)$ and $\CC(I)$. Thus, each residually connected geometry over $I$ corresponds to a unique residually connected chamber system over $I$ (up to isomorphism), with same automorphism group, and vice versa.
\end{theorem}

\subsection{Hypermaps}
Hypermaps were introduced in~\cite{cori}, see also~\cite{dessins, JonesSingerman}.
A hypermap  $\mathcal{M}$ is a transitive permutation representation $\Delta\rightarrow \textup{Sym}(\Sigma)$ of the group $$\Delta=\left\langle r_0,r_1,r_2~|~r_0^2=r_1^2=r_2^2=1\right\rangle\cong C_2*C_2*C_2,$$
where $\Sigma$ is a set, the elements of which are called flags. One can think of each flag as consisting of a hypervertex, a hyperedge and a hyperface. If the automorphism group $\textup{Aut}(M)$ of the hypermap acts transitively on the flags, the hypermap is called regular. In that case, the hypervertices, the hyperedges and the hyperfaces are the orbits of the dihedral subgroups $\langle r_1,r_2 \rangle$, $\langle r_2, r_0\rangle$ and $\langle r_0, r_1\rangle$, respectively and the hypermap is said to be of \textit{type} $(p,q,r)$ where $p$ is the order of $r_1r_2$, $q$ is the order of $r_2r_0$ and $r$ is the order of $r_0r_1$.

The coset geometry constructed from the automorphism group $\textup{Aut}(M)$ and the three dihedral subgroups is a rank 3 incidence geometry on the type set $\{\mbox{hypervertex, hyperedge, hyperface}\}$. A maximal flag (a chamber) of this incidence geometry is a triple consisting of one element of each type such that they are pairwise incident. 

It is important to note that the notion of a flag is different for hypermaps and incidence geometries; in a hypermap flag the incidence between the objects is a ternary relation, while the incidence relation in the coset geometry is binary. This implies that the coset geometry constructed from a regular hypermap can have more flags than the hypermap. In particular, the transitive action of the automorphism group on the flags of the hypermap does not imply the transitive action of the automorphism group on the chambers of the corresponding coset geometry. See for instance~\cite[Example 4.4]{hypertopes}.

\section{The Chamber System approach}\label{sec:main}
Let $e>1$ be a positive integer such that $3\mid 2e+1$ and let $q = 2^{2e+1}$.
Let $G\cong \Sz(q)$ be the Suzuki group acting on the $q^2+1$ points of the ovoid $\mathcal O$ (see section \ref{sec:Suzuki}).
The aim of this section is to construct a chamber system over $I:=\{0,1,2\}$ from $G$.
We first fix a few conventions. Regular letters like $q,p$ and $m$ designate integers. Capital letters like $P,Q,R$ designate points of $\OO$. Greek letters like $\rho,\gamma$ and $\tau$ designate group elements, with $\tau$ restricted to outer automorphisms of $G$ of order $3$.

Since $3\mid 2e+1$, the underlying field $\FF_q$ is a extension of degree 3 of a subfield $\FF_{q_0}$ and, therefore, there are outer automorphisms of $G$ of order $3$.
We will call any such automorphism a triality of $G$. Let $\tau$ be a triality of $G$ and let $P,Q,R$ be a triple of points of $\mathcal O$ permuted cyclically by $\tau$. Since the triality $\tau$ comes from the action of a power of the Frobenius automorphism $\Phi$ on the ovoid $\OO$, we can consider $\tau$ as a permutation of  the points $\OO$. With that in mind, for any point $P \in \OO$, $\tau(P)$ is just the image of $P$ by that permutation, while for an element $\gamma \in G$, we will denote by $\tau(\gamma) := \tau \gamma \tau^{-1} = \gamma^\tau$, the element of $G$ that sends any $X \in \OO$  to $\tau \circ \gamma \circ \tau^{-1} (X)$. As such, if $\gamma$ fixes some point $P \in \OO$, then $\tau(\gamma)$ fixes $\tau(P)$.

\begin{lemma}\label{lem:existence}
    Let $P$ be a point of $\OO$ and let $\tau$ be a triality of $G$ such that $\tau(P) \neq P$. Let $A$ be the set of involutions of $G$ with fixed point $P$. Then the set $\{\rho \cdot \tau(\rho) \mid \rho \in A\}$ contains only elements of odd order, no two of which are conjugate. 
\end{lemma}

\begin{proof}
    All elements of even order of $G$ have exactly one fixed point. Since $\rho$ and $\tau(\rho)$ are always involutions with different fixed points, their product has either $0$ or $2$ fixed points, and must thus have odd order.

    It remains to show that, if $\rho \neq \rho'\in A$, then $\rho \tau(\rho)$ cannot be conjugated to $\rho'\tau(\rho')$.

    Let $H := G_{P,\tau^2(P)}$ be the stabilizer in $G$ of $p$ and $\tau^2(P)$. Hence, $H$ is a cyclic group of order $q-1$. Let $\zeta$ be a generator of $H$.
    For $\omega := \zeta^i$, $i = 1,2, \cdots, q-2$, we have that

    \begin{equation*}
        (\rho \tau(\rho))^\omega = \rho^\omega \tau(\rho)^\omega = \rho^\omega \omega \tau(\rho) \omega^{-1} = \rho^\omega \tau(\tau^{-1}(\omega) \rho \tau^{-1} (\omega^{-1}) ) = \rho^\omega \tau(\rho ^{\tau^{-1}(\omega)})
    \end{equation*}

    We claim that $\rho^\omega \neq \rho^{\tau^{-1}(\omega)}$.
    Indeed, suppose that $\rho^\omega = \rho^{\tau^{-1}(\omega)}$. Conjugating both sides by $\omega^{-1}$, we obtain that $\rho = \rho^{\tau^{-1}(\omega)\omega^{-1}}$. Set $\beta = \tau^{-1}(\omega)\omega^{-1} = \tau^2(\omega)\omega^{-1}$. 
    We have that $\rho^\beta= \rho$ thus $\rho^\beta$ fixes $P$ and $\beta$ fixes $P$ as well. This means that $\tau^{-1}(\omega)$ must fix $P$ since $\omega^{-1}$ and $\beta$ fix $P$. 
    But $\tau^{-1}(\omega)$ also fixes $\tau^{-1}(P)= \tau^2(P)$ and $\tau^{-1}(\tau^2(P))= \tau(P)$ so $\tau^{-1}(\omega)$ would have three fixed points, a contradiction.

    We conclude using Lemma~\ref{lem:conjClasses}. Indeed we showed that $\rho\tau(\rho)$ is conjugated to $\rho^\omega \tau(\rho ^{\tau^{-1}(\omega)})$ which in turn cannot be conjugated to $\rho^\omega \tau(\rho ^\omega)$ by Lemma~\ref{lem:conjClasses}, unless $\rho^\omega = \rho ^{\tau^{-1}(\omega)}$, which we ruled out. Since $H$ acts transitively on $A$, this concludes the lemma.
\end{proof}

\begin{theorem}
[Existence of base Chamber]\label{cor:existence}
    Let $m$ be any odd integer that arises as the order of some element of $G$. For any point $P$ of $\OO$ and for any triality $\tau$ of $G$ such that $\tau(P) \neq P$, there exists an involution $\rho_0$ fixing $P$ such that $\rho_0 \cdot \tau(\rho_0)$ has order $m$. Moreover, if $\langle \rho_0\tau(\rho_0) \rangle\neq \langle \rho_0\tau(\rho_0) \rangle^\tau$ then $G$ is generated by $\rho_0, \tau(\rho_0)$ and $\tau^2(\rho_0)$.
\end{theorem}

\begin{proof}
By~\cite[§11, Theorem 5]{Suzuki1962}, the Suzuki group $G$ has exactly $q-1$ conjugacy classes of elements of odd order.
The first part of the statement is thus a direct consequence of Lemma~\ref{lem:existence} as  there are exactly $q-1$ involutions fixing a given point $P$.

 Let  $m$ be the order of $\rho_0\tau(\rho_0)$, let $H = \langle \rho_0, \tau(\rho_0), \tau^2(\rho_0) \rangle$, and suppose that $H$ is a proper subgroup of $G$. Then $H$ must be contained in some maximal subgroup of $G$. We use the classification of maximal subgroups provided by Theorem \ref{thm:maximalsubgroups}. Suppose first that $H$ is contained in a maximal subgroup $M$, isomorphic to either $D_{2(q-1)}$, $C_{\alpha_{q}}:C_4$. or $\C_{\beta_{q}}:C_4$. By the last sentence of Theorem \ref{thm:maximalsubgroups}, there is a unique cyclic subgroup $K$ of order $m$ in $M$. Hence, the three subgroups $\langle \rho_0\tau(\rho_0) \rangle, 
\langle \tau(\rho_0)\tau^2(\rho_0)\rangle$ and $\langle \tau^2(\rho_0)\rho_0
\rangle$ all contain $K$. But any two of these subgroups also have an involution in common, so they must then all be equal, a contradiction. 

Suppose instead that $M$ is isomorphic to a maximal Suzuki subgroup $\Sz(q_0)$. If $H =M$, since $H$ is always stabilized by $\tau$, it would imply that $M$ is also stabilized by $\tau$. This means that $\tau$ restricts to an automorphism of $\Sz(q_0)$. This implies that $q_0 = 2^{2e_0+1}$ with $2e_0+1$ divisible by $3$ and therefore, by maximality, $q_0 = q^{1/3}$. But the only Suzuki subgroup $\Sz(q^{1/3})$ stabilized by $\tau$ is precisely the one fixed element-wise by $\tau$. This is a contradiction since $\tau$ clearly moves $\rho_0$.

If instead $H$ is again a proper subgroup of $M$, we can repeat the same argument starting with $\Sz(q_0)$ this time. At some point, we arrive at a contradiction, hence showing that $G = \langle \rho_0,\tau(\rho_0),\tau^2(\rho_0) \rangle$.
\end{proof}

Using Theorem \ref{cor:existence}, we can find the base chamber of the chamber system we want to construct. We first investigate some geometric properties of such base chambers and show their algebraic implications.

Let $\rho_0$ be an involution with fixed point $P$ such that the order of $(\rho_0\tau(\rho_0))$ is equal to $m$ for some integer $m$. 
Let $\rho_1 := \tau(\rho_0)$ and $\rho_2 := \tau^2(\rho_0)$. Then $\rho_1$ has fixed point $Q = \tau(P)$ and $\rho_2$ has fixed point $R = \tau^2(P)$. For $i,j,k$ such that $I = \{i,j,k\},$ the subgroups generated by $\langle \rho_j, \rho_k \rangle$ are dihedral groups $H_i = H_i(\rho_0,\tau)$ of order $2m$. The existence of such a configuration is guaranteed by Corollary~\ref{cor:existence}.

Let $\OO_i = \OO(H_i)$ be the set of points of $\OO$ that are fixed by some involution in $H_i$. Thus $\OO_i$ contains exactly $m$ points for each $i \in I$. 

\begin{definition}\label{def:triangle}
    Let $\OO_i$ be set of points of $\OO$ for $i \in I$. We say that an ordered triple $( \OO_0, \OO_1, \OO_2 )$ forms a \textit{triangle} if $\OO_i \cap \OO_j \neq \emptyset$ for all $i \neq j = 0,1,2$. We say that the triangle $( \OO_0, \OO_1, \OO_2 )$ is \textit{non-degenerate} if the $\OO_i$'s are all distinct for $i = 0,1,2$. We say that the triangle $( \OO_0, \OO_1, \OO_2 )$ is \textit{proper} if $\OO_0 \cap \OO_1 \cap \OO_2 = \emptyset$.
    Finally, if $\rho_0$ is an involution of $G$ and $\tau$ is a triality as above, we say that $T_{\rho_0,\tau} := (\OO_0 = \OO(H_0(\rho_0,\tau)),\OO_1 = \OO(H_1(\rho_0,\tau)),\OO_2=\OO(H_2(\rho_0,\tau)))$ is the {\em triangle associated to} $(\rho_0,\tau)$.
\end{definition}

We now show that non-degeneracy of $T_{\rho_0,\tau}$ implies that $G$ is generated by $\rho_0, \tau(\rho_0)$ and $\tau^2(\rho_0)$.

\begin{lemma}\label{lem:degenerate}
    If $T_{\rho_0, \tau}$ is degenerate, then $H_0(\rho_0,\tau) = H_1(\rho_0,\tau) =H_2(\rho_0,\tau)$
\end{lemma}

\begin{proof}
    Suppose that $\OO_0 = \OO_1 = \OO_2 $ but that the three subgroups $H_0, H_1$ and $H_2$ are distinct. By the second part of Theorem \ref{cor:existence}, we can then deduce that $G = \langle \rho_0, \tau(\rho_0), \tau^2(\rho_0) \rangle$. But all three of $
    \rho_0, \tau(\rho_0)$ and $\tau^2(\rho_0)$ stabilize $\OO_0$. Therefore, $\langle \rho_0, \tau(\rho_0), \tau^2(\rho_0) \rangle$ is a subgroup of the stabilizer in $G$ of $\OO_0$, which is a proper subgroup since $G$ acts transitively on the whole ovoid $\OO$. This is a contradiction.  
\end{proof}
\begin{corollary}
    If $T_{\rho_0, \tau}$ is non-degenerate, then $G = \langle \rho_0, \tau(\rho_0), \tau^2(\rho_0) \rangle$.
\end{corollary}
\begin{proof}
    By Lemma \ref{lem:degenerate}, we know that the subgroups $H_0,H_1,H_2$ are distinct. We can then conclude by using the second part of Theorem \ref{cor:existence}.
\end{proof}
\begin{figure}
\begin{center} \label{TriangleOvoids}
    \begin{tikzpicture}
    \filldraw[black] (-2,0) circle (2pt)  node[anchor=north]{$P$};
    \filldraw[black] (2,0) circle (2pt)  node[anchor=north]{$\tau(P)$};
    \filldraw[black] (0,3.46410161514) circle (2pt)  node[anchor=north]{$\tau^2(P)$};
    \draw [rotate = 60](1,1.73205) ellipse (3cm and 1cm);
    \draw [rotate = -60](-1,1.73205) ellipse (3cm and 1cm);
    \draw (0,0) ellipse (3cm and 1cm);
    \filldraw[black] (0,0) circle (0pt)  node[anchor=north]{$O_2$};
    \filldraw[black] (-1,1.73205) circle (0pt)  node[anchor=south]{$O_1$};
    \filldraw[black] (1,1.73205) circle (0pt)  node[anchor=south]{$O_0$};
    \end{tikzpicture}
    \caption{The base chamber $T_{\rho_0,\tau}$ of the chamber system}
\end{center}
\end{figure}
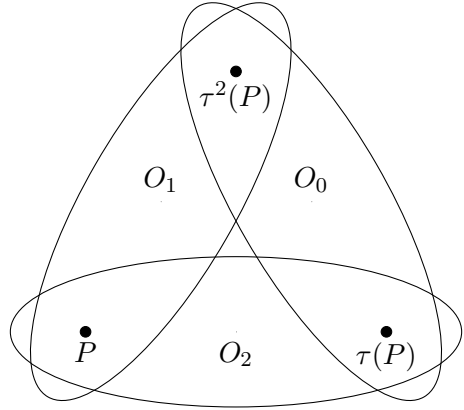

The next Lemma gives criterion for $T_{\rho_0,\tau}$ to be non-degenerate or proper.

\begin{lemma} \label{lem:order5}
    If $\rho_0$ in an involution and $\tau$ is a triality such that $\rho_0 \tau(\rho_0)$ is of order $m$ with $ m \equiv 2 (\bmod\; 3)$, then the triangle $T_{\rho_0,\tau} =(\OO_0,\OO_1,\OO_2)$ associated to $(\rho_0,\tau)$ is non-degenerate. Furthermore, if $m = 5$, the triangle $T_{\rho_0,\tau} =(\OO_0,\OO_1,\OO_2)$ is also proper.
\end{lemma}
\begin{proof}
    Since $\OO_i \cap \OO_j$ always contains the fixed point of $\rho_k$,
    the triangle associated to any pair $(\rho_0,\tau)$ is always a triangle in the sense of Definition \ref{def:triangle}.     Suppose now that $\OO_0 = \OO_1 = \OO_2$. Then the triality $\tau$ acts on $\OO_0$ by permuting its points. Then, since $|\OO_0| = m = 2 (\bmod\; 3)$, the triality $\tau$ must fix at least $2$ points $X$ and $Y$ of $\OO_0$. Since $\tau$ also sends $H_0$ to itself by Lemma \ref{lem:degenerate}, we conclude that $\tau$ must fix $\rho_X$ and $\rho_Y$, the two involutions of $H_0$ having respectively $X$ and $Y$ as fixed points. But $H_0$ is generated by $\rho_X$ and $\rho_Y$, so $\tau$ should act trivially on $H_0$, a contradiction. 
    This shows that the triangle is non-degenerate.

    Suppose now that $m=5$ and that $\OO_0 \cap \OO_1 \cap \OO_2$ is not empty and contains some point $X \in \OO$. Let $P$ be a point in $\OO_0 \cap \OO_1$ which is not $X$. Let $\rho_X^0$ be the involution fixing $X$ in $H_0$, $\rho_X^1$ be the involution fixing $X$ in $H_1$ and $\rho$ be the involution fixing $P$ in $H_0 \cap H_1$. Then both $\rho\rho_X^0$ and $\rho\rho_X^1$ have order $5$. By Lemma \ref{lem:conjClasses} $\rho\rho_X^0$ and $\rho\rho_X^1$ cannot be conjugated. But there is only one conjugacy class of elements of order $5$ in $G$, a contradiction.
    This shows that the triangle is proper, and concludes the proof.
\end{proof}

We are now ready to show that for some well chosen value of $m = O(\rho_0\tau(\rho_0))$, we can construct regular hypermaps of type $(m,m,m)$ by letting $G$ acts on $T_{\rho_0,\tau}$.

\begin{corollary}\label{cor:hypermaps}
    Any triangle $T_{\rho_0,\tau}$ with $\rho_0\tau(\rho_0)$ of order $m$ and $m \equiv 2 (\bmod\; 3)$ defines a hypermap of type $(m,m,m)$ with automorphism group $G$. In particular, for any $q = 2^{2e+1}$ with $2e+1$ divisble by 3, there exists an hypermap of type $(\alpha_{q},\alpha_{q},\alpha_{q})$ where $\alpha_{q} = 2^{2e+1} + 2^{e+1} +1$ or of type $(\beta_{q},\beta_{q},\beta_{q})$ where $\beta_{q} = 2^{2e+1} -2^{e+1} +1$, depending on the parity of $e$.
\end{corollary}
\begin{proof}
    Lemma \ref{lem:order5} tells us that $G = \langle \rho_0, \tau(\rho_0), \tau^2(\rho_0) \rangle $ whenever the order of $\rho_0\tau(\rho_0)$ is equal to $2$ modulo $3$. Notice that $ 2^{2e+1} + 2^{e+1} +1 \equiv 2 \pm 2 +1 (\bmod\; 3)$ where the sign depends on whether $e$ is odd or even. So for every value of $e$, exactly one of $\alpha_{q}$ and $\beta_{q}$ is congruent to $2 \bmod 3$ while the other is congruent to $1 \bmod 3$. Choosing the right one then yields a regular hypermap of the wanted type.
\end{proof}

Let $\mathcal{M}$ be one of the hypermaps obtained by Corollary \ref{cor:hypermaps}. The triality $\tau$ then sends $\mathcal{M}$ to a map $\tau(\mathcal{M})$ which is isomorphic to $\mathcal{M}$, but where the roles of vertices, hyperedges and hyperfaces are permuted cyclically. Such operation of the finite order on hypermaps have been studied in \cite{jones2010hypermap}. An operation of order three is called a triality and an operation of order two is called a duality. Since $\Out(G)$ has no elements of order $2$, the map $\mathcal{M}$ cannot have dualities. The hypermaps of Corollary \ref{cor:hypermaps} thus admit trialities but no dualities.

We now proceed, after a few results, to define the chamber system $C = \{\gamma(T_{\rho_0,\tau}) \mid \gamma\in G \}$. From here on, we suppose that we have a pair $(\rho_0,\tau)$ such that the associated triangle $T_{\rho_0,\tau}$ is proper and non-degenerate. 

\begin{lemma}\label{lem:tripleUniqueness}
    Let $(P,Q,R)$ be a triple of points permuted by a triality $\tau$. Then, $G_{\{P,Q,R\}}$ is trivial.
\end{lemma}

\begin{proof}
    Suppose there is a $\gamma \in G$ preserving the set $\{P,Q,R\}$. If $\gamma$ fixes $\{P,Q,R\}$ pointwise, it must be the identity as the only element of $G$ that fixes three points of $\OO$ is the identity element. Also, $\gamma$ cannot permute $P,Q$ and $R$ cyclically, since $3$ does not divide the order of $G$. Suppose, without loss of generality, that $\gamma(P) = P$ and that $\gamma$ exchanges $Q$ and $R$. Then $\tau(\gamma)$ fixes $Q$ and exchanges $P$ and $R$. But then one easily checks that $\gamma \cdot \tau(\gamma)$ permutes cyclically $P,Q$ and $R$, a contradiction.
\end{proof}

For a triangle $(\OO_0,\OO_1,\OO_2)$, its image $\gamma(\OO_0,\OO_1,\OO_2)$ by the action of $\gamma \in G$ is the triangle $(\gamma(\OO_0),\gamma(\OO_1),\gamma(\OO_2))$.

\begin{lemma}\label{lem:triangles}
    The image $\gamma(T_{\rho_0,\tau})$ by $\gamma \in G$ of the triangle associated to $(\rho_0,\tau)$ is $T_{\rho_0^\gamma, \tau^\gamma}$, the triangle associated to $(\rho_0^\gamma, \tau^\gamma)$.
\end{lemma}

\begin{proof}
    In other words, we need to show that $\OO_i(\rho_0^\gamma, \tau^\gamma) = \gamma(\OO_i(\rho_0,\tau))$. By definition, $\OO_2(\rho_0^\gamma, \tau^\gamma)$ is the set of fixed points of the involutions of $\langle \rho_0^\gamma, \tau^\gamma \rho_0^\gamma (\tau^\gamma) ^{-1}\rangle = \langle \rho_0, \tau \rho_0 \tau^{-1}\rangle ^\gamma$. Since $\OO_2(\rho_0,\tau)$ is the set of fixed points of involutions in $\langle \rho_0, \tau \rho_0 \tau^{-1} \rangle$, this concludes the Lemma for $\OO_2$. The cases of $\OO_0$ and $\OO_1$ are identical.
\end{proof}

\begin{corollary} [Chamber-transitivity] \label{cor:FT}
    Let $C = \{\gamma(T_{\rho_0,\tau}) \mid \gamma \in G \} = \{(\gamma(\OO_0),\gamma(\OO_1),\gamma(\OO_2)) \mid \gamma \in G\}$. Then $|C| = |G|$. Moreover, if $UC := \{ \{\gamma(\OO_0),\gamma(\OO_1),\gamma(\OO_2)\}\mid \gamma \in G\}$ is the set of unordered triples, we also have that $|UC| = |G|$.
\end{corollary}
\begin{proof}
    We only prove the second statement as it implies the first. Let $c = T_{\rho_0,\tau}$ and $c' \in C$. Then $c' = \gamma(c)$ for some $\gamma\in G$. 
    Suppose that $c'= c$. Then $\gamma$ must send the triple $\{\OO_0,\OO_1,\OO_2\}$ to itself. This also means that $\gamma$ must conjugate the triple $\{H_0,H_1,H_2\}$ to itself. Since $H_i \cap H_j = \langle \rho_k \rangle$ for $\{i,j,k\} = I$, we conclude that $\gamma$ must conjugate the triple $\{\rho_0,\rho_1,\rho_2\}$ to itself. Therefore, if $P,Q,R$ are the fixed points of $\rho_i$, we can conclude that $P,Q,R$ must be fixed setwise by $\gamma$, and thus that $\gamma = 1_G$, by Lemma \ref{lem:tripleUniqueness}.
\end{proof}

Let $C = \{\gamma(T_{\rho_0,\tau}) \mid \gamma \in G \}$ and let $d \in C$. We say that $d$ is a \textit{chamber} of $\C$. Then $d = T_{\rho^\gamma, \tau^\gamma}$ for some $\gamma \in G$ by Lemma \ref{lem:triangles}. Let $H_i(d)$ be the unique dihedral subgroup of order $2m$ inside of the stabilizer of $\OO_i(\rho^\gamma, \tau^\gamma)$. Then $\rho_i(d)$ is defined to be the unique involution in $H_j(d) \cap H_k(d)$, whenever $\{i,j,k\} = I$.

\begin{definition}[The chamber system]
    Let $C = \{\gamma(T_{\rho_0,\tau}) \mid \gamma\in G \}$ as before and define equivalence relations $\sim_i, i \in I$ on $C$ as follows:
    for $c_1,c_2 \in C$, 
    $c_1 \sim_i c_2$ if and only if $c_1 = \rho_i(c_1)(c_2)$ or $c_1 = c_2$.
\end{definition}
Note that if $c_1 = \rho_i(c_1)(c_2)$, then $\rho_i(c_1) = \rho_i(c_2)$ so that $\sim_i$ is indeed symmetric.

\begin{theorem}\label{thm:FT}
    $G$ acts chamber transitively by automorphisms on the chamber system $(C,\sim_i)$, $i \in I$.
\end{theorem}

\begin{proof}
    Chamber-transitivity is clear by Corollary \ref{cor:FT}. To check that $G$ acts by automorphisms, it suffices to check that the following diagram commutes:
\begin{center}
 \begin{tikzcd}
C \arrow[r, "\rho_i"] \arrow[d, "\gamma"'] & C \arrow[d, "\gamma"] \\
C \arrow[r, "\rho_i"]                 & C               
\end{tikzcd}
\end{center}

To check that, let $c \in C$. Then we need to compare $(\gamma \circ \rho_i(c))(c)$ to $(\rho_i(\gamma(c)) \circ \gamma)(c)$. But since $\rho_i(\gamma(c)) = \rho_i(c)^\gamma$, the two expressions are indeed equal.
\end{proof}

We have thus constructed so far a chamber system $(C,\sim_i)$, $i \in I$ on which $G$ acts chamber transitively by automorphisms. It remains to show that  $(C,\sim_i)$, $i \in I$ is residually connected. For that we first need a technical result.

\begin{lemma}\label{lem:normalizer}
    Let $K$ be a cyclic subgroup of $G =\Sz(q)$ of prime odd order $p$. Let 
    $N_G(K)$ be the normalizer in $G$ of $K$ and let $H$ be a dihedral subgroup of $G$ of order $2p$. Then $H$ contains $K$ if and only if $H \leq N_G(K)$.  
\end{lemma}

\begin{proof}

    Let $H$ be a dihedral subgroup containing $K$. Then $H$ normalizes $K$ as $K$ is of odd order. So $H \leq N_G(K)$.

    Suppose that $H\leq N_G(K)$ and $K\not\leq H$.
    Then $N_G(K)$ contains two cyclic subgroups of order $p$, namely $K$ and the subgroup of order $p$ of $H$.
Also $N_G(K) = D_{2(q-1)}$ if $p| q-1$ and $N_G(K) = C_{x}:C_4$ otherwise (where $x = q+\sqrt{2q}+1$ or $q-\sqrt{2q}+1$). In any case, $N_G(K)$ has a unique cyclic subgroup of odd order $p$, a contradiction.
\end{proof}

We are now ready to show that $(C,\sim_i)$, $i \in I$ is residually connected, under the assumption that $\gamma(T_{\rho_0,\tau})$ is proper and non-degenerate.

\begin{theorem}~\label{th3.11}
Let $C = \{\gamma(T_{\rho_0,\tau}) \mid \gamma\in G \}$ for a proper and non-degenerate triangle $\gamma(T_{\rho_0,\tau})$. Suppose moreover that $\rho_0\tau(\rho_0)$ is an element of prime order $p$.
The chamber system $(C,\sim_i), i \in I$ is residually connected.
\end{theorem}
\begin{proof}
 Since $G$ acts chamber transitively by automorphisms on  $(C,\sim_i)$, by Proposition \ref{prop:CosetChamberSystem}, the coset chamber system $(G,G^i)$ is equivalent to  $(C,\sim_i)$, where $G^i = \langle \rho_i \rangle$ where $\rho_0$ is the involution used to define the triangle $T_{\rho_0,\tau}$ and $\rho_1 = \tau(\rho_0)$, $\rho_2 = \tau^2(\rho_0)$.  Therefore, we can use Theorem \ref{thm:chamberRC}. We will check that for all $J,K,L \subseteq I$, we have 
\begin{equation} \label{equation:RC}
    G^{(L)} \cap G^{(J)}G^{(K)} = G^{(L\cap J)}G^{(L\cap K)}.
\end{equation}

Note that $G^{(L)} \cap G^{(K)} = G^{(L \cap K)}$ for all $L,K \subset I$ and that $G^{(K)}G^{(L)} = G^{(L)}$ whenever $K \subset L$.
Most of the cases hold in a quite straightforward manner. For example, if $J = I$, equation~(\ref{equation:RC}) becomes 
\begin{equation*} 
    G^{(L)} \cap G^{(K)} = G^{(L)}G^{(L\cap K)} = G^{(L\cap K)}.
\end{equation*}
The case of $J = I$ or $K = I$ is similar.

Suppose now that $J = \emptyset$. Then equation~(\ref{equation:RC}) becomes 
\begin{equation*} 
     G^{(L)} \cap G^{(K)} = G^{(L\cap K)}.
\end{equation*}
The same happens if $K = \emptyset$ or $L = \emptyset$.

If $J = \{ 0 \}, K = \{1\}$ and $L = \{ 2\}$ then equation~(\ref{equation:RC}) becomes 
\begin{equation*} 
     \langle \rho_2 \rangle \cap \langle \rho_0 \rangle\langle \rho_1 \rangle = \{ \Id \}
\end{equation*}
which holds since $\rho_2$ is not contained in $H_2$ since $(\rho_0,\tau)$ is non-degenerated.

The harder case arises when $|J| =|K| =|L| = 2$. So suppose that $J = \{0,1\}, K = \{1,2\}$ and $L = \{0,2\}$. Equation~(\ref{equation:RC}) then becomes
    \begin{equation} \label{eq:FT}
    \langle \rho_0,\rho_2 \rangle \cap \langle \rho_0,\rho_1\rangle\langle \rho_1,\rho_2\rangle = \langle \rho_0\rangle\langle \rho_2\rangle = \{e,\rho_0,\rho_2,\rho_0\rho_2\}
\end{equation}
which is the same as 
\begin{equation*} 
    H_1 \cap H_2 H_0 = \{e,\rho_0,\rho_2,\rho_0\rho_2\} \subset H_1.
\end{equation*}

Note that $\{e,\rho_0,\rho_2,\rho_0\rho_2\} \subseteq H_1 \cap H_2H_0$ is evidently true, so it suffices to show the opposite inclusion.
To show this, let $\gamma_2 \in H_2$ and $\gamma_0 \in H_0$. 

Suppose first that $\gamma_2$ and $\gamma_0$ are both involutions. If they have the same fixed point $x$, then $\gamma_2\gamma_0$ is also an involution with fixed point $x$ which can then not be an element of $H_1$ since $(\rho_0,\tau)$ is non-degenerated. So suppose that $\gamma_2$ and $\gamma_0$ have different fixed points. If $\gamma_2 = \rho_0$ and $\gamma_0 = \rho_2$, their product is $\rho_0\rho_2$ so we are done. So we can suppose that either $\gamma_2 \neq \rho_0$ or $\gamma_0 \neq \rho_2$. Suppose that $\gamma_2 \neq \rho_0$. The product $\gamma_2\gamma_0$ cannot be an involution since it has $0$ or $2$ fixed points. Hence, either $\gamma_2\gamma_0$ is not in $H_1$ or it is of order $p$. If it is of order $p$, then the subgroup generated by $\gamma_2$ and $\gamma_0$ is a dihedral  subgroup $H$ of order $2p$ such that $H \cap H_1 = \langle \rho_0\rho_2 \rangle$. Therefore, by Lemma \ref{lem:normalizer}, $H$ must be contained in $N = N_G(\langle \rho_0\rho_2 \rangle)$. By the same Lemma, we know that $H_2$ cannot be in $N$. But that means that $\gamma_2 \notin N$ since $H_2 = \langle \rho_0,\gamma_2 \rangle$ and $\rho_0$ clearly is in $N$. This is a contradiction, which shows that $\gamma_2\gamma_0 \notin H_1$ as long as $\gamma_2 \neq \rho_0$, or $\gamma_0 \neq \rho_2$.

Suppose now that both $\gamma_2 \in H_2$ and $\gamma_0 \in H_0$ have order $p$. Then $\gamma_2 = (\rho_0\rho_1)^k$ for some $k = 1, \ldots, p-1$ and $\gamma_0 = (\rho_1\rho_2)^m$ for $m=1, \ldots, p-1$. Therefore,
    \begin{equation*}
        \gamma_2\gamma_0 = (\rho_0\rho_1)^k(\rho_1\rho_2)^m = (\rho_0\rho_1)^{k-1}(\rho_0\rho_2)(\rho_1\rho_2)^{m-1}
    \end{equation*}
    If $k=m=1$ then $\gamma_2\gamma_0 = \rho_0\rho_2 \in H_1$. Else, note that $(\rho_0\rho_1)^{k-1}\rho_0 = \rho_0^h$ with $h \in H_2$ if $k$ is odd and $(\rho_0\rho_1)^{k-1}\rho_0 = \rho_1^h$ with $h$ in $H_2$ if $k$ is even. The same thing holds for $\rho_2(\rho_1\rho_2)^{m-1}$. This means that if at least one of $k$ and $m$ is not $1$, $\gamma_2\gamma_0$ becomes equal to the product of two involutions, one in $H_2$ and one in $H_0$. We have already shown that such a product cannot be in $H_1$.
    We note that, from the above argument, it can be deduced that if $h_i$ is an element of order $m$ in $H_i$ and $h_j$ is an element of order $m$ in $H_j$, then if $h_ih_j$ is in $H_k$, it must be of order $p$.

Finally suppose without loss of generality that $\gamma_2$ has order $p$ and $\gamma_0$ has order $2$. Suppose that $\gamma_2\gamma_0 \in H_1$. If $\gamma_0 = \rho_2$ then $\gamma_0\in H_1$ and $\gamma_2 = \gamma_2(\gamma_0)^2\in H_1$ as well, a contradiction with Lemma~\ref{lem:normalizer}. Suppose thus that $\gamma_0 \neq \rho_0$. If the order of $\gamma_2\gamma_0$ is $2$, then $\langle \gamma_2, \gamma_2\gamma_0\rangle$ is a dihedral group of order $2p$ and we can use the same method as the first case. If the order of $\gamma_2\gamma_0$ is $p$, then $\gamma_2$ is of order $p$ in $H_2$, $\gamma_2\gamma_0$ is of order $p$ is $H_1$ so, by the comment at the end of the previous case, $\gamma_0 = \gamma_2^{-1} (\gamma_2\gamma_0)$ should be of order $p$, a contradiction.
\end{proof}

Finally, let $C = \{\gamma(T_{\rho_0,\tau}) \mid \gamma\in G \}$ be a chamber system obtained by Theorem \ref{th3.11}, and let $\Gamma = (G,(G_i)_{i \in I}$ be the associated coset geometry. We want to show that $\Gamma$ admits trialities but no dualities. The triality $\tau$ naturally induces a triality of $\Gamma$ since it permutes the three generators $\rho_0,\rho_1$ and $\rho_3$. To prove that $\Gamma$ admits no dualities, we show a results that associates to any correlation of $\Gamma$ an automorphism of $G$.

Let $\Gamma$ be a thin, simply flag transitive incidence geometry, let $G = \Aut(\Gamma)$ and let $c$ be a chamber of $\Gamma$. Let $\{c_0,\cdots, c_n\}$ be the set of chambers of $\Gamma$ incident to $c$. For each $c_i$, there is a unique element $\gamma_i \in G$ such that $c_i = \gamma_i(c)$. Let $S = \{\gamma_0, \cdots, \gamma_n \}$. We then have that each $\gamma_i$ has order $2$ and that $G = \langle \gamma_0, \cdots, \gamma_n \rangle$. For any correlation $\alpha$ of $\Gamma$ we can definite a map $\phi_\alpha \colon S \to S$ as follows.
For each $i$, there is exactly one chamber $c_i$ which is $i$-adjacent to $c$. The image $\alpha(c_i)$ is a chamber which is $j-$adjacent to $\alpha(c)$ for some $j \in I$. We then define $\phi_\alpha(\gamma_i) := \gamma_j$ and note $j = \alpha(i)$. In other words, $\phi_\alpha(\gamma_i) = \gamma_{\alpha(i)}$ where we let $\alpha$ naturally act on the type set $I$.

\begin{proposition}\label{prop:CorToAut}
Let $\Gamma$ be a thin, simply flag transitive incidence geometry and let $\alpha$ be a correlation of $\Gamma$. Then the map $\phi_\alpha \colon S \to S$ naturally extends to an  automorphism $\Phi_\alpha$ of $G$. 
\end{proposition}

\begin{proof}
To show that $\Phi_\alpha$ is a homomorphism, it suffices to show that any word $w$ in the alphabet $S$ representing the identity in $G$ is sent by $\alpha$ to a word $\alpha(w)$ which also represents the identity in $G$. Let $w$ be such a word. Then it uniquely defines a loop $l$ with base point $c$ in the chamber graph $CG(\Gamma)$. This loop can be recorded by a sequence $i_1,i_2,\cdots, i_n$, $i_j \in I$, of adjacencies. The image $\alpha(l)$ of this loop is once again a loop, now with base point $\alpha(c)$, which is recorded by the sequence $\alpha(i_1), \alpha(i_2), \cdots, \alpha(i_n)$. This loop $\alpha(l)$ also uniquely defines a word which by definition is $\alpha(w)$. Thus $\alpha(w)$ represents the identity and $\Phi_\alpha$ is a homomorphism.

The homomorphism $\Phi_\alpha$ is surjective since any element $\gamma \in G$ can be represented by a path $p$ in $CG(\Gamma)$. Any path $p'$ such that $\alpha(p') = p$ represents an elements of $\gamma'\in G$ such that $\alpha(\gamma') = \gamma$.

The homomorphism $\Phi_\alpha$ is injective since $\alpha$ is injective on the vertices of $CG(\Gamma)$. Therefore, if $p$ is a path in $CG(\Gamma)$ which is not a loop, the image $\alpha(p)$ cannot be a loop so that any word not representing the identity in $G$ cannot be sent to the identity by $\alpha$.

\end{proof}

\begin{theorem}\label{thm:main}
    Let $G = \Sz(q)$ with $q = 2^{2e+1}$ where $2e+1$ is a multiple of $3$. There exists a flag transitive, residually connected, thin incidence geometry $\Gamma$ over $I = \{0,1,2\}$ with diagram 
    
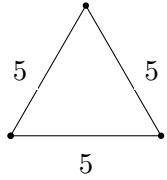
\begin{center}
    \begin{tikzpicture}[scale = 0.5]
    
   \filldraw[black] (-2,0) circle (2pt)  node[anchor=north]{};
   \filldraw[black] (2,0) circle (2pt)  node[anchor=north]{};
    \draw (-2,0) -- (2,0) node [midway, below = 3pt, fill=white]{$ \; \; \; \; \; 5 \; \; \; \; \; $};
       \filldraw[black] (-2,0) circle (2pt)  node[anchor=north]{};
   \filldraw[black] (0,3.46410161514) circle (2pt)  node[anchor=north]{};
    \draw (-2,0) -- (0,3.46410161514) node [midway, left = 4pt, fill=white]{$  5  $};
    \draw (2,0) -- (0,3.46410161514) node [midway, right = 4pt, fill=white]{$ 5 \ $};
    \end{tikzpicture}
    \end{center}
    such that $\Aut(\Gamma)\cong \Sz(q)$ has index 3 in $Cor(\Gamma)\cong \Sz(q):C_3$. Moreover, $\Gamma$ admits trialities but no dualities.

\end{theorem}

\begin{proof}
By Lemma \ref{lem:order5}, we know that the triangle $T_{\rho_0,\tau}$ associated to any pair $(\rho_0,\tau)$ with $\rho_0\tau(\rho_0)$ of order $5$ is always proper and non-degenerate. Theorems~\ref{thm:FT} and ~\ref{th3.11} then show that we can construct a chamber transitive and residually connected chamber system $C$ with automorphism group $G$.
Theorem \ref{thm:Correspondence}  implies the existence of a coset geometry $\Gamma$ with automorphism group $G$ and parabolic subgroups $H_i = \langle \rho_j,\rho_k \rangle,$ for $\{i,j,k\} = \{0,1,2\}$. The triality $\tau$ then acts on $\Gamma$ as a correlation of order three. Proposition \ref{prop:CorToAut} guarantees that $\Gamma$ cannot have dualities. Indeed, if a duality were to exist, the map $\Phi_\alpha$ fixing $\rho_i$ and exchanging $\rho_j$ and $\rho_k$ should be an automorphism of $G$. This automorphism cannot be inner, by Lemma \ref{lem:tripleUniqueness}. It also cannot be outer since $G$ has no outer automorphism of order $2$.
\end{proof}
\section{Concluding remarks}\label{sec:conclusion}

We conclude this article with a few remarks. In this article, we decided to first construct a chamber system with the desired properties and then to use Theorem \ref{thm:Correspondence} to obtain the desired incidence geometries. There are two main reasons for that decision. Indeed, the chamber system approach not only simplifies some technical difficulties, as explained in more details below, but it also allows the construction to remain more concrete and geometric.

Suppose we were to try to construct directly and geometrically the incidence geometries of Theorem \ref{thm:main}. For that, we need to find a subset $S$ of points of $\OO$ such that the stabilizer of $S$ under the action of $\Sz(q)$ is a dihedral group of order $10$. This means we cannot simply set $S = \OO_5$, a sub-ovoid of $5$ points, as we did in the chamber system approach. Indeed the stabilizer in $\Sz(q)$ of $S$ would then be a subgroup isomorphic to $\Sz(2)$. The group $\Sz(2)$ nonetheless contains a dihedral group of order $10$ as an index $2$ subgroup, so it is most likely possible to fix this problem, but it requires additional effort.

On the other hand, if we accept to loose the geometric interpretation, the incidence geometries of Theorem \ref{thm:main} can be constructed directly as coset incidence geometries in the following way. Let $T_{\rho_0,\tau}$ is a proper and non-degenerate triangle and let $H_0,H_1$ and $H_2$ be the $3$ dihedral subgroups associated to $T_{\rho_0,\tau}$. We can then construct the coset geometry $\Gamma(G,(H_0,H_1,H_2))$. Equation (\ref{eq:FT}) is satisfied, meaning that this coset geometry is flag transitive.
Indeed, by a result of Tits~\cite[Section 1.4]{Tits1974}, this equation implies that any triple of cosets that are pairwise incident have a common element, and therefore that $G$ acts transitively on the chambers of $\Gamma$. Residual connectedness is straightforward.

\bibliographystyle{abbrv} 
\bibliography{SuzukiTrialities}

\begin{thebibliography}{10}

\bibitem{buekenhout2013diagram}
F.~Buekenhout and A.~M. Cohen.
\newblock {\em Diagram geometry: related to classical groups and buildings},
  volume~57.
\newblock Springer Science \& Business Media, 2013.

\bibitem{cori}
R.~Cori.
\newblock {\em Un code pour les graphes planaires et ses applications}.
\newblock Ast\'{e}risque, No. 27. Soci\'{e}t\'{e} Math\'{e}matique de France,
  Paris, 1975.
\newblock With an English abstract.

\bibitem{downs2016mobius}
M.~Downs and G.~A. Jones.
\newblock M{\"o}bius inversion in {S}uzuki groups and enumeration of regular
  objects.
\newblock In {\em Symmetries in Graphs, Maps, and Polytopes: 5th SIGMAP
  Workshop, West Malvern, UK, July 2014}, pages 97--127. Springer, 2016.

\bibitem{hypertopes}
M.~E. Fernandes, D.~Leemans, and A.~I. Weiss.
\newblock Highly symmetric hypertopes.
\newblock {\em Aequationes Math.}, 90(5):1045--1067, 2016.

\bibitem{JonesSingerman}
G.~Jones and D.~Singerman.
\newblock Maps, hypermaps and triangle groups.
\newblock In {\em The {G}rothendieck theory of dessins d'enfants ({L}uminy,
  1993)}, volume 200 of {\em London Math. Soc. Lecture Note Ser.}, pages
  115--145. Cambridge Univ. Press, Cambridge, 1994.

\bibitem{jones2010hypermap}
G.~A. Jones and D.~Pinto.
\newblock Hypermap operations of finite order.
\newblock {\em Discrete mathematics}, 310(12):1820--1827, 2010.

\bibitem{dessins}
G.~A. Jones and J.~Wolfart.
\newblock {\em Dessins d'enfants on {R}iemann surfaces}.
\newblock Springer Monographs in Mathematics. Springer, Cham, 2016.

\bibitem{Leemans1998}
D.~Leemans.
\newblock Thin geometries for the {S}uzuki simple group {${\rm Sz}(8)$}.
\newblock {\em Bull. Belg. Math. Soc. Simon Stevin}, 5(2-3):373--387, 1998.

\bibitem{LeemansStokes2019}
D.~Leemans and K.~Stokes.
\newblock Coset geometries with trialities and their reduced incidence graphs.
\newblock {\em Acta Math. Univ. Comenian. (N.S.)}, 88(3):911--916, 2019.

\bibitem{leemans2022incidence}
D.~Leemans and K.~Stokes.
\newblock Incidence geometries with trialities coming from maps with {Wilson}
  trialities.
\newblock {\em Innov. Incidence Geom.}, 20(2--3):325--340, 2023.

\bibitem{luneburg2012translation}
H.~L{\"u}neburg.
\newblock {\em Translation planes}.
\newblock Springer-Verlag, 1980.

\bibitem{Suzuki1962}
M.~Suzuki.
\newblock On a class of doubly transitive groups. {II}.
\newblock {\em Ann. of Math. (2)}, 79:514--589, 1964.

\bibitem{Tits1957}
J.~Tits.
\newblock Sur les analogues alg\'{e}briques des groupes semi-simples complexes.
\newblock In {\em Colloque d'alg\`ebre sup\'{e}rieure, tenu \`a {B}ruxelles du
  19 au 22 d\'{e}cembre 1956}, Centre Belge de Recherches Math\'{e}matiques,
  pages 261--289. \'{E}tablissements Ceuterick, Louvain, 1957.

\bibitem{Tits1962}
J.~Tits.
\newblock Ovoïdes à translations.
\newblock {\em Rend. Mat. Appl.}, 5(21):37–59, 1962.

\bibitem{Tits1974}
J.~Tits.
\newblock {\em Buildings of spherical type and finite {BN}-pairs}.
\newblock Lecture Notes in Mathematics, Vol. 386. Springer-Verlag, Berlin-New
  York, 1974.

\bibitem{tits1982}
J.~Tits.
\newblock A local approach to buildings.
\newblock In {\em The geometric vein}, pages 519--547. Springer, New
  York-Berlin, 1981.

\end{thebibliography}
\end{document}